\documentclass[12pt,reqno]{amsart}
\usepackage{graphicx}
\usepackage{float}
\usepackage{enumerate}
\usepackage{mathtools}
\usepackage{breqn}
\usepackage{array, geometry, graphicx}
\usepackage{amsmath,amsfonts,amssymb,amsthm,cite,mathrsfs}
\usepackage{caption}
\newtheorem{theorem}{Theorem}[section]
\newtheorem{corollary}[theorem]{Corollary}
\newtheorem{lemma}[theorem]{Lemma}

\theoremstyle{definition}

\theoremstyle{remark}
\newtheorem{remark}[theorem]{Remark}
\numberwithin{equation}{section}
\DeclareMathOperator{\RE}{Re}

\makeatletter
\@namedef{subjclassname@2010}{%
  \textup{2010} Mathematics Subject Classification}
\makeatother

\allowdisplaybreaks

\begin{document}

\title{Sufficient conditions and radius problems for the Silverman class}
\author[S. Sivaprasad Kumar]{S. Sivaprasad Kumar}

\address{Department of Applied Mathematics, Delhi Technological University,
Delhi--110 042, India}
\email{spkumar@dce.ac.in}
\author[P. Goel]{Priyanka Goel}

\address{Department of Applied Mathematics, Delhi Technological University,
Delhi--110 042, India}
\email{priyanka.goel0707@gmail.com}

\begin{abstract}
For $0<\alpha\leq1$ and $\lambda>0,$ let
\begin{equation}\label{1}
  G_{\lambda,\alpha}=\left\{f\in\mathcal{A}:\left|\dfrac{1-\alpha+\alpha zf''(z)/f'(z)}{zf'(z)/f(z)}-(1-\alpha)\right|<\lambda, z\in\mathbb{D}\right\},
\end{equation}
the general form of Silverman class introduced by Tuneski and Irnak. For this class we derive some sufficient conditions in the form of differential inequalities. Further, we consider the class $\Omega,$ given by
\begin{equation}\label{omega}
	\Omega=\left\{f\in\mathcal{A}:|zf'(z)-f(z)|<\dfrac{1}{2},\;z\in\mathbb{D}\right\}.
\end{equation}
For the above two classes, we establish inclusion relations involving some other well known subclasses of $\mathcal{S}^*$  and find radius estimates for different pairs involving these classes.
\end{abstract}



 \maketitle

\section{Introduction}
Let $\mathbb{D}:=\{z:|z|<1\}$ be the open unit disk and $\mathcal{H}$ be the class of all analytic functions defined on $\mathbb{D}.$ In addition, let $\mathcal{A}_n$ be the class of all normalized analytic functions of the form $f(z)=z+a_{n+1}z^{n+1}+a_{n+2}z^{n+2}+\cdots$ with $\mathcal{A}:=\mathcal{A}_1.$ Denote by $\mathcal{S},$ the subclass of $\mathcal{A}$ consisting of univalent functions. Let $\mathcal{S}^*$ and $\mathcal{C}$ denote the class of starlike and convex functions respectively. For two analytic functions $f$ and $F,$ it is said that $f$ is subordinate to $F$, denoted by $f\prec F$ if there exists a Schwarz function $\omega$ such that $f(z)=F(\omega(z)).$ Let $\Phi_M$ denotes the Ma-Minda class, consisting of the functions $\phi$ satisfying the following properties: (i)~$\phi$ is analytic and univalent;
(ii)$\phi$ is symmetric with respect to real axis;
(iii)$\phi$ has positive real part in $\mathbb{D}$;
(iv)$\phi$ is starlike with respect to $\phi(0)=1$;
(v)$\phi'(0)>0.$
For $\phi\in\Phi_M,$ Ma and Minda~\cite{maminda} introduced a general subclass of $\mathcal{S}^*,$ defined as the class of all the functions $f\in\mathcal{A}$ such that $zf'(z)/f(z)\prec \phi(z),$ denoted by $\mathcal{S}^*(\phi).$ In later years, many authors came up with different subclasses of $\mathcal{S}^*,$ which they defined by taking $\phi$ as a particular Ma-Minda function. Some of the classes which are used in the present work are listed as follows: the class $\mathcal{S}^*_{L},$ introduced by Sok\'{o}\l~\cite{sokol}, where $\phi$ is taken as $\sqrt{1+z};$ the class $\mathcal{S}^*_{e}$ with $\phi(z)=e^z$ defined by Mendiratta et al.~\cite{mendiratta}; the class $\mathcal{S}^*_{RL}$ introduced by Mendiratta et al.~\cite{mendiratta2}; the class $\mathcal{S}^*_{C}$ introduced by~Sharma et al.~\cite{cardiod}, where $\phi(z)=1+4z/3+2z^2/3;$ the class $\mathcal{S}^*_{S}$ introduced by Cho et al.~\cite{vktsine}, where $\phi(z)=1+\sin{z};$ the class $\mathcal{S}^*_{Cr}$ associated with the crescent $z+\sqrt{1+z^2}$, introduced by Sharma et al.~\cite{crescent}; the class $\mathcal{S}^*_{SG}$ introduced by Goel and Kumar~\cite{first} with $\phi(z)=2/(1+e^{-z});$ the class $\mathcal{S}^*_{\wp}$ introduced by Kumar and Gangania~\cite{kamal}, where $\phi$ represents a cardioid given by $1+ze^z;$ the class $\mathcal{S}^*_{Ne}$ introduced by Wani and Swaminathan~\cite{lateef}, where $\phi$ is taken as $1+z-z^3/3.$ In the year 1999, Silverman~\cite{silver1} introduced the following class:
\begin{equation*}
  G_b=\left\{f\in\mathcal{A}:\left|\dfrac{1+zf''(z)/f'(z)}{zf'(z)/f(z)}-1\right|<b,\;z\in\mathbb{D}\right\},\quad b>0.
\end{equation*}
The author established conditions on $b$ for which the class $G_b$ is contained in the class of starlike functions and further in the class $\mathcal{S}^*(\alpha),$ the class of starlike functions of order $\alpha.$ In addition, the author estimated the largest radius for which every starlike function of order $1/2$ belongs to $G_b.$ The proofs of these theorems are based on the properties of Schwarz function. Some of the properties are given as follows:
\begin{lemma}[Schwarz-Pick Lemma]\cite{steven}
 Let $\omega$ be a function analytic on $\mathbb{D}$ such that $|\omega(z)|\leq 1$ and $\omega(0)=0$, then for all $z\in\mathbb{D}$
  \begin{equation*}
    |\omega'(z)|\leq \dfrac{1-|\omega(z)|^2}{1-|z|^2}.
  \end{equation*}
\end{lemma}
\begin{lemma}\label{lem2}\cite{dieudonne}
   Let $\omega:\mathbb{D}\rightarrow\mathbb{D}$ be analytic, then for all $z\in\mathbb{D}$
   \begin{equation*}
   |\omega'(z)|\leq \begin{cases}
     1, &|z|\leq \sqrt{2}-1\\
     \dfrac{(1+r^2)^2}{4r(1-r^2)}, & |z|\geq\sqrt{2}-1.
   \end{cases}
   \end{equation*}
\end{lemma}
Besides these inequalities, Dieudonn\'e proved a number of others relating to derivatives of Schwarz function~\cite{dieudonne}. In 2006, the class $G_b$ was generalized by Tuneski and Irnak~\cite{tuneski} in the form given by~\eqref{1}. By taking $\alpha=1/2,$ $G_{\lambda,\alpha}$ reduces to the class $G_b$ with $b=2\lambda.$ In 2017, Peng and Zhong~\cite{zeng17} introduced a new subclass $\Omega$ of $\mathcal{A}$ given by~\eqref{omega}.
For this class, the authors proved that $f\in\Omega$ is equivalent to saying that
\begin{equation}\label{t15}
	f(z)=z+\dfrac{1}{2}z\int_0^z \varphi(\zeta)d\zeta,
\end{equation}
where $\varphi$ is analytic in $\mathbb{D}$ and $|\varphi(z)|\leq1,\;z\in\mathbb{D}.$ They also proved its inclusion in $\mathcal{S}^*,$ estimated radius of convexity and discussed many other properties of $\Omega$. In 2019, Peng and Obradovi\'c~\cite{zeng19} estimated logarithmic and inverse coefficients, proved Robertson's $1/2$ conjecture and 1/2 theorem and other results related to Hadamard product and coefficient multipliers. Later in this year, Wani and Swaminathan~\cite{wani} defind a new class $\Omega_n=\{f\in\mathcal{A}_n:|zf'(z)-f(z)|<1/2,\;z\in\mathbb{D}\}.$ They obtained sufficient conditions for $\Omega_n,$ proved inclusion properties of $\Omega$ and derived sharp radii estimates for different subclasses of $\mathcal{S}^*.$ Motivated by their work, we consider similar problems for the class $G_{\lambda,\alpha}.$ We allot double integral functions to this class by proving sufficient conditions and utilizing the conditions to construct such functions. Further by using the concept of subordination, we prove several inclusion relations between the class $G_{\lambda,\alpha},$ $\Omega$ and other well known subclasses of $\mathcal{S}^*$ mentioned above. We also obtain some radius estimates for the functions belonging to the different forms of $\mathcal{S}^*(\phi)$ ensuring that they are contained in $\Omega$ as well as $G_{\lambda,\alpha}.$
\section{Main Results}
\begin{theorem}\label{tthm1}
Let $f\in\mathcal{A}_n,$ $0\leq \alpha<1$ and $\lambda>0.$ If
\begin{equation}\label{t2}
  \left|zf''(z)-\alpha\left(f'(z)-\dfrac{f(z)}{z}\right)\right|< \delta,
\end{equation}
where $\delta$ is the smallest positive root of
\begin{equation}\label{phir}
\phi(r):=(1 + n) (2 \alpha n-\lambda(n+1) - n)r^2+n (1 - \alpha + n) (2\lambda(n+1)+n+\alpha n^2)r-\lambda n^2(n+1-\alpha)^2,
\end{equation}
then $f\in G_{\lambda, \alpha}.$
\end{theorem}
\begin{proof}
From~\eqref{t2}, we have
\begin{equation*}
  zf''(z)-\alpha\left(f'(z)-\dfrac{f(z)}{z}\right)\prec \delta z, \quad z\in\mathbb{D}.
\end{equation*}
Let $P(z)=f'(z)-f(z)/z,$ then $P(0)=0$ and
\begin{equation*}
  (1-\alpha)P(z)+zP'(z)=zf''(z)-\alpha\left(f'(z)-\dfrac{f(z)}{z}\right)\prec \delta z.
\end{equation*}
Now applying~\cite[Theorem~3.1b]{ds} for $h(z)=\delta z/(1-\alpha)$ and $\gamma= 1-\alpha,$ we obtain
\begin{equation*}
  P(z)\prec \dfrac{\delta z}{n+1-\alpha},
\end{equation*}
which is equivalent to
\begin{equation}\label{t3}
 f'(z)-\dfrac{f(z)}{z}\prec \dfrac{\delta z}{n+1-\alpha}.
\end{equation}
Now let us suppose $p(z)=f(z)/z,$ then from~\eqref{t3}
\begin{equation}\label{t4}
  zp'(z)= f'(z)-\dfrac{f(z)}{z}\prec \dfrac{\delta z}{n+1-\alpha}.
\end{equation}
Now by using~\cite[Lemma~8.2a]{ds}, we get
\begin{equation*}
  p(z)=\dfrac{f(z)}{z}\prec 1+\dfrac{\delta z}{n(n+1-\alpha)},
\end{equation*}
which further yields the following inequality
\begin{equation}\label{t6}
  1-\dfrac{\delta}{n(n+1-\alpha)}<\left|\dfrac{f(z)}{z}\right|<1+\dfrac{\delta}{n(n+1-\alpha)}.
\end{equation}
From~\eqref{t4}, it is clear that
\begin{equation}\label{t5}
\left|f'(z)-\dfrac{f(z)}{z}\right|<\dfrac{\delta}{n+1-\alpha},
\end{equation}
which further implies
\begin{equation}\label{t7}
|f'(z)|>\left|\dfrac{f(z)}{z}\right|-\dfrac{\delta}{n+1-\alpha},
\end{equation}
From~\eqref{t6} and~\eqref{t7}, we may conclude that
\begin{equation}\label{t8}
  |f'(z)|>1-\dfrac{\delta(n+1)}{n(n+1-\alpha)}.
\end{equation}
From~\eqref{t2}, we have
\begin{equation}\label{t9}
 \left|f'(z)\left(\dfrac{zf''(z)}{f'(z)}\right)-\alpha\left(f'(z)-\dfrac{f(z)}{z}\right)\right|<\delta.
\end{equation}
Now from~\eqref{t8} and~\eqref{t9}, we observe that
\begin{equation}\label{t29}
  \left(1-\dfrac{\delta(n+1)}{n(n+1-\alpha)}\right)\left|\dfrac{zf''(z)}{f'(z)}\right|<|f'(z)|\left|\dfrac{zf''(z)}{f'(z)}\right|<\delta+\alpha\left|f'(z)-\dfrac{f(z)}{z}\right|.
\end{equation}
Here we may note that $\delta$ is the smaller of the two roots of $\phi(r),$ which is given by~\eqref{phir}. So we get
\begin{equation*}
  \delta=\dfrac{n(n+1-\alpha)(n+\alpha n^2+2\lambda(n+1)-\sqrt{n^2+\alpha^2n^4+2\alpha n^3+8\alpha\lambda n+12\alpha\lambda n^2+4\alpha\lambda n^3})}{2(n+1)(\lambda(n+1)+n-2\alpha n)}.
\end{equation*}
Since
\begin{eqnarray*}
  &(n+\alpha n^2+2\lambda(n+1)-\sqrt{n^2+\alpha^2n^4+2\alpha n^3+8\alpha\lambda n+12\alpha\lambda n^2+4\alpha\lambda n^3})(n+\alpha n^2+2\lambda(n+1)\\
  &+\sqrt{n^2+\alpha^2n^4+2\alpha n^3+8\alpha\lambda n+12\alpha\lambda n^2+4\alpha\lambda n^3})=4\lambda(n+1)(\lambda(n+1)+n-2\alpha n),
\end{eqnarray*}
we have
\begin{eqnarray*}
  \delta&=&\dfrac{2\lambda n(n+1-\alpha)}{n+\alpha n^2+2\lambda(n+1)+\sqrt{n^2+\alpha^2n^4+2\alpha n^3+8\alpha\lambda n+12\alpha\lambda n^2+4\alpha\lambda n^3}}
  \leq \dfrac{2\lambda n(n+1-\alpha)}{2\lambda(n+1)}.
\end{eqnarray*}
Therefore
\begin{equation*}
  1-\dfrac{\delta(n+1)}{n(n+1-\alpha)}>0
\end{equation*}
and thus~\eqref{t29} implies
\begin{equation}\label{t10}
 \left|\dfrac{zf''(z)}{f'(z)}\right|<\dfrac{\delta+\dfrac{\alpha\delta}{n+1-\alpha}}{1-\dfrac{\delta(n+1)}{n(n+1-\alpha)}}=\dfrac{n(n+1)\delta}{n(n+1-\alpha)-\delta(n+1)}.
\end{equation}
Now let us consider the following inequality
\begin{align*}
 \left(1-\dfrac{\delta(n+1)}{n(n+1-\alpha)}\right)\bigg|\alpha\dfrac{f(z)f''(z)}{(f'(z))^2}-&(1-\alpha)+(1-\alpha)\dfrac{f(z)}{zf'(z)}\bigg|\\
  &< |f'(z)|\left|\alpha\dfrac{f(z)f''(z)}{(f'(z))^2}-(1-\alpha)+(1-\alpha)\dfrac{f(z)}{zf'(z)}\right|\\
  &=\left|\alpha\dfrac{f(z)f''(z)}{f'(z)}-(1-\alpha)\left(f'(z)-\dfrac{f(z)}{z}\right)\right|\\
 &<\alpha\left|\dfrac{f(z)}{z}\right|\left|\dfrac{zf''(z)}{f'(z)}\right|+(1-\alpha)\left|f'(z)-\dfrac{f(z)}{z}\right|.
\end{align*}
Using~\eqref{t6},~\eqref{t5} and~\eqref{t10} in the above inequality, we get
\begin{align*}
\bigg(1-&\dfrac{\delta(n+1)}{n(n+1-\alpha)}\bigg)\left|\alpha\dfrac{f(z)f''(z)}{(f'(z))^2}-(1-\alpha)+(1-\alpha)\dfrac{f(z)}{zf'(z)}\right|\\
 & <\alpha\left(1+\dfrac{\delta}{n(n+1-\alpha)}\right)\left(\dfrac{n(n+1)\delta}{n(n+1-\alpha)-\delta(n+1)}\right)+(1-\alpha)\left(\dfrac{\delta}{n+1-\alpha}\right)=:\tau,
\end{align*}
which implies
\begin{eqnarray*}
\left|\alpha\dfrac{f(z)f''(z)}{(f'(z))^2}-(1-\alpha)+(1-\alpha)\dfrac{f(z)}{zf'(z)}\right|&<&\left(\dfrac{n(n+1-\alpha)}{n(n+1-\alpha)-\delta(n+1)}\right)\tau\\
&=&\lambda.
\end{eqnarray*}
Thus we have
\begin{equation*}
  \left|\dfrac{1-\alpha+\alpha zf''(z)/f'(z)}{zf'(z)/f(z)}-(1-\alpha)\right|<\lambda
\end{equation*}
and the result follows.
\end{proof}
\begin{corollary}
  Let $0\leq\alpha<1,$ $\lambda>0$ and $g\in\mathcal{H}.$ If $|g(z)|<\delta,$ where $\delta$ is the smallest positive root of $\phi(r):=(1 + n) (2 \alpha n-\lambda(n+1) - n)r^2+n (1 - \alpha + n) (2\lambda(n+1)+n+\alpha n^2)r-\lambda n^2(n+1-\alpha)^2,$ then
  \begin{equation*}
    f(z)=z+z^{n+1}\int_0^1\int_0^1 g(rsz)r^{n-\alpha}s^{n-1}drds
  \end{equation*}
  is in $G_{\lambda,\alpha}.$
\end{corollary}
\begin{proof}
Suppose that $f(z)$ satisfies the following differential equation
\begin{equation}\label{t11}
 z f''(z)-\alpha\left(f'(z)-\dfrac{f(z)}{z}\right)=z^ng(z).
\end{equation}
Let
\begin{equation*}
  H(z)= f'(z)-\dfrac{f(z)}{z},
\end{equation*}
then from~\eqref{t11}, we have
\begin{equation*}
  (1-\alpha)H(z)+zH'(z)=z^n g(z).
\end{equation*}
Now applying~\cite[Theorem~3.1b]{ds}, we obtain the solution of the above differential equation, given by
\begin{equation*}
  H(z)=\dfrac{1}{z^{1-\alpha}}\int_0^z g(t)t^{n-\alpha}dt.
    \end{equation*}
Now if we substitute $t=rz$ in the above equation, then
\begin{equation*}
  H(z)=z^n\int_0^1 g(rz)r^{n-\alpha}dr,
\end{equation*}
Taking $h(z)=f(z)/z,$ we have
\begin{equation*}
  zh'(z)=f'(z)-\dfrac{f(z)}{z}=H(z).
\end{equation*}
Now by using~\cite[Lemma~8.2a]{ds}, we obtain
\begin{equation*}
  h(z)=1+\int_0^z \dfrac{H(t)}{t}dt.
\end{equation*}
Substituting $t=sz$ yields
\begin{eqnarray*}
h(z)&=&1+\int_0^1 \dfrac{H(sz)}{s}ds\\
& =&1+\int_0^1\left(\dfrac{(sz)^n}{s}\int_0^1 g(rsz) r^{n-\alpha}dr\right)ds\\
& =& 1+z^n\int_0^1\int_0^1 g(rsz) r^{n-\alpha}s^{n-1}drds.
\end{eqnarray*}
Thus
\begin{equation*}
  f(z)=z+z^{n+1}\int_0^1\int_0^1 g(rsz) r^{n-\alpha}s^{n-1}drds.
\end{equation*}
Now using Theorem~\ref{tthm1} along with the fact that $|g(z)|<\delta$, we have $f\in G_{\lambda,\alpha}.$
\end{proof}
\begin{corollary}
Let $f\in\mathcal{A}$ satisfies
\begin{equation}\label{t14}
  \left|zf''(z)-\dfrac{1}{2}\left(f'(z)-\dfrac{f(z)}{z}\right)\right|<\dfrac{3}{8}(5-\sqrt{21}),
\end{equation}
then $z(zf'(z)/f(z))$ is univalent in $\mathbb{D}.$
\end{corollary}
\begin{proof}
If we take $n=1,\;\alpha=1/2$ and $\delta=3(5-\sqrt{21})/8$ in Theorem~\ref{tthm1}, then~\eqref{t14} implies that $f\in G_{\tiny{\frac{1}{4},\frac{1}{2}}}.$ We know that $G_{\tiny{\frac{1}{4},\frac{1}{2}}}=G_{\tiny{\frac{1}{2}}}$ and thus by using~\cite[Theorem~2]{tunobr}, the result follows.
\end{proof}
\begin{theorem}\label{tthm2}
Let $f\in\mathcal{A}_n,$ $0\leq\alpha<1$ and $\lambda>0.$ If
\begin{equation}\label{t12}
  |zf''(z)-\alpha(f'(z)-1)|<\dfrac{\delta(n+1)(n-\alpha)}{\alpha+(n+1)(n-\alpha)}\quad z\in\mathbb{D},
\end{equation}
where $\delta$ is the smallest positive root of $\phi(r):=(1 + n) (2 \alpha n-\lambda(n+1) - n)r^2+n (1 - \alpha + n) (2\lambda(n+1)+n+\alpha n^2)r-\lambda n^2(n+1-\alpha)^2,$ then $f\in G_{\lambda, \alpha}.$
\end{theorem}
\begin{proof}
From~\eqref{t12}, we have for $z\in\mathbb{D}$
\begin{equation*}
 zf''(z)-\alpha(f'(z)-1)\prec  \dfrac{\delta(n+1)(n-\alpha)z}{\alpha+(n+1)(n-\alpha)}.
\end{equation*}
Let $P(z)=f'(z)-(1+\alpha)f(z)/z$, then
\begin{equation*}
  P(z)+zP'(z)=zf''(z)-\alpha f'(z)\prec \dfrac{\delta(n+1)(n-\alpha)z}{\alpha+(n+1)(n-\alpha)}-\alpha.
\end{equation*}
Using Lemma~\cite[Theorem~3.1b]{ds}, we have
\begin{equation*}
  P(z)\prec \dfrac{\delta(n-\alpha)z}{\alpha+(n+1)(n-\alpha)}-\alpha,
\end{equation*}
which further implies
\begin{equation*}
  f'(z)-(1+\alpha)\dfrac{f(z)}{z}\prec \dfrac{\delta(n-\alpha)z}{\alpha+(n+1)(n-\alpha)}-\alpha.
\end{equation*}
Now let us take
\begin{equation*}
  p(z)=\dfrac{f(z)}{z}-1\quad\text{and}\quad q(z)=\dfrac{\delta z}{\alpha+(n+1)(n-\alpha)}.
\end{equation*}
It is easy to observe that $q(0)=0,\;q'(0)\neq 0$ and $\RE\left(1+\tfrac{zq''(z)}{q'(z)}\right)=1>\tfrac{\alpha}{n}.$
Next, we observe
\begin{eqnarray*}
  zp'(z)-\alpha p(z)=f'(z)-(1+\alpha)\dfrac{f(z)}{z}+\alpha\prec \dfrac{\delta(n-\alpha)z}{\alpha+(n+1)(n-\alpha)}=nzq'(z)-\alpha q(z).
\end{eqnarray*}
Then by using Lemma~\cite[Lemma~8.2a]{ds}, we obtain
\begin{equation*}
  \dfrac{f(z)}{z}-1=p(z)\prec q(z)=\dfrac{\delta z}{\alpha+(n+1)(n-\alpha)},
\end{equation*}
which implies
\begin{equation}\label{t13}
  \left|\dfrac{f(z)}{z}-1\right|<\dfrac{\delta}{\alpha+(n+1)(n-\alpha)}.
\end{equation}
Finally from~\eqref{t12} and~\eqref{t13}, we have
\begin{eqnarray*}
  \left|zf''(z)-\alpha\left(f'(z)-\dfrac{f(z)}{z}\right)\right| &\leq& \left|zf''(z)-\alpha(f'(z)-1)\right|+\alpha\left|\dfrac{f(z)}{z}-1\right|  \\
   &<&  \dfrac{\delta(n+1)(n-\alpha)}{\alpha+(n+1)(n-\alpha)}+\dfrac{\alpha\delta}{\alpha+(n+1)(n-\alpha)}\\
   &=&\delta.
\end{eqnarray*}
Applying Theorem~\ref{tthm1}, the result follows.
\end{proof}
\begin{corollary}
Let $0\leq\alpha<1,$ $\lambda>0$ and $g\in\mathcal{H}.$ If
\begin{equation*}
  |g(z)|<\dfrac{\delta(n+1)(n-\alpha)}{\alpha+(n+1)(n-\alpha)}, \quad z\in\mathbb{D},
\end{equation*}
where $\delta$ is the smallest positive root of $\phi(r):=(1 + n) (2 \alpha n-\lambda(n+1) - n)r^2+n (1 - \alpha + n) (2\lambda(n+1)+n+\alpha n^2)r-\lambda n^2(n+1-\alpha)^2,$  then
  \begin{equation*}
    f(z)=z+z^{n+1}\int_0^1\int_0^1 g(rsz)r^{n-1-\alpha}s^{n}drds
  \end{equation*}
  is in $G_{\lambda,\alpha}.$
\end{corollary}
\begin{proof}
  Suppose $f\in\mathcal{A}_n$ satisfies
  \begin{equation*}
    zf''(z)-\alpha(f'(z)-1)=z^n g(z).
  \end{equation*}
  Taking $H(z)=f'(z)-1,$ the above equation reduces to
  \begin{equation*}
    zH'(z)-\alpha H(z)=z^n g(z).
  \end{equation*}
By using~\cite[Theorem~3.1b]{ds}, we obtain the solution of the above differential equation as follows
\begin{equation*}
  H(z)=z^{\alpha}\int_0^z g(t)t^{n-\alpha-1}dt.
\end{equation*}
Taking $t=rz,$ it reduces to
\begin{equation*}
  H(z)=z^n\int_0^1 g(rz)r^{n-\alpha-1}dr
\end{equation*}
and thus
\begin{equation*}
  f(z)=z+z^{n+1}\int_0^1\int_0^1 g(rsz)r^{n-1-\alpha}s^{n}drds.
\end{equation*}
By Theorem~\ref{tthm2}, the result follows.
\end{proof}
\begin{theorem}\label{theod}
Let $f\in G_{\lambda,\alpha}\;(\lambda>0,\;1/3<\alpha\leq1)$. Then $zf'(z)/f(z)\prec 1/(1\pm cz),$ where $c=\lambda/(3\alpha-1)$ and the result is sharp.
\end{theorem}
\begin{proof}
  Let $p(z)=zf'(z)/f(z)=1/(1+c \omega(z)).$ Then
 \begin{eqnarray*}
   \left|\dfrac{1-\alpha+\alpha zf''(z)/f'(z)}{zf'(z)/f(z)}-(1-\alpha)\right|&=&\left|\dfrac{1-2\alpha}{p(z)}+\dfrac{\alpha zp'(z)}{p^2(z)}+2\alpha-1\right|\\
   &=&\left|(1-2\alpha)c\omega(z)-\alpha c z\omega'(z)\right|.
 \end{eqnarray*}
 Now we show that $|\omega(z)|<1$ for $z\in\mathbb{D}.$ Suppose on contrary there exists a point $z_0\in\mathbb{D}$ such that $|\omega(z_0)|=1$ and $z_0\omega'(z_0)=k\omega(z_0)(k\geq 1).$ Then
 \begin{eqnarray*}
  \left|\dfrac{1-2\alpha}{p(z_0)}+\dfrac{\alpha zp'(z_0)}{p^2(z_0)}+2\alpha-1\right|&=&|\omega(z_0)c(1-\alpha(k+2))|\\
  &=& \left|\dfrac{\lambda}{1-3\alpha}(1-\alpha(k+2))\right|\\
  &>& \lambda,
 \end{eqnarray*}
 which is a contradiction to the assumption that $f\in G_{\lambda,\alpha}.$ For the function $f(z)=z/(1\pm cz),$ we obtain that $zf'(z)/f(z)=1/(1\pm cz)$ and
 $$\left|\dfrac{1-2\alpha}{p(z)}+\dfrac{\alpha zp'(z)}{p^2(z)}+2\alpha-1\right|=\lambda.$$
\end{proof}
\begin{remark}
For $\alpha=1/2,$ $G_{\lambda,\alpha}$ reduces to the class $G_b$ defined by Silverman and the result above reduces to~\cite[Theorem 1]{tunobr} with $b=2\lambda.$
\end{remark}
\begin{remark}
For $\alpha=1,$ $G_{\lambda,\alpha}$ reduces to the class $G_{\lambda,1}$ defined by Tuneski and the result above reduces to~\cite[Theorem 3.1]{bul} with $h(z)=\lambda z.$
\end{remark}
\begin{theorem}\label{incomega}
Let $\lambda>0$ and $1/3<\alpha<1$ be such that $\lambda<(2-\sqrt{3})(3\alpha-1).$ Then $G_{\lambda,\alpha}\subset\Omega.$
\end{theorem}
\begin{proof}
  Let $f\in G_{\lambda,\alpha},$ then Theorem~\ref{theod} implies that
  \begin{equation*}
    \dfrac{zf'(z)}{f(z)}\prec \dfrac{1}{1+cz}=:\phi_0(z),\quad\text{where}\; c=\dfrac{\lambda}{3\alpha-1}.
  \end{equation*}
By the structural formula, we know that $f\in\mathcal{S}^*(\phi_0)$ if and only if there exists a function $\phi(z)\prec \phi_0(z)$ such that
\begin{equation*}
  f(z)=z\exp{\int_0^z \dfrac{\phi(t)-1}{t}dt}.
\end{equation*}
Taking $\phi(z)=\phi_0(z),$ we obtain the extremal function for the class $\mathcal{S}^*(\phi_0),$ given by $f_0(z)=z/(1+cz).$ Then by the growth theorem, we have $|f(z)|\leq f_0(r)$ on $|z|=r.$ Hence
\begin{equation*}
  |zf'(z)-f(z)|=|f(z)|\left|\dfrac{zf'(z)}{f(z)}-1\right|\leq |f_0(1)|\left|\dfrac{-cz}{1+cz}\right|\leq \dfrac{c}{(1-c)^2}.
\end{equation*}
We have $c=\lambda/(3\alpha-1)<2-\sqrt{3},$ which further implies that
\begin{equation*}
  |zf'(z)-f(z)|\leq \dfrac{c}{(1-c)^2}<\dfrac{1}{2}
\end{equation*}
and the proof is complete.
\end{proof}
\begin{lemma}\label{incgen}
  Let $\lambda>0,\;1/3<\alpha<1$ and $\phi\in\Phi_M$ with $\phi(\mathbb{D})=\Delta.$ Then $G_{\lambda,\alpha}\subset \mathcal{S}^*(\phi),$ whenever $(1+r_1)\lambda<(3\alpha-1)r_1,$ where $r_1$ is the radius of the largest disk contained in $\Delta$ and centered at $1.$
\end{lemma}
\begin{proof}
  Let $f\in G_{\lambda,\alpha}.$ Then from the proof of Theorem~\ref{incomega}, we have
  \begin{equation*}
    \left|\dfrac{zf'(z)}{f(z)}-1\right|<\dfrac{c}{1-c},\quad\text{with}\;c=\dfrac{\lambda}{3\alpha-1}.
  \end{equation*}
 Since $(1+r_1)\lambda<(3\alpha-1)r_1,$ we have
  \begin{equation*}
    \left|\dfrac{zf'(z)}{f(z)}-1\right|<\dfrac{c}{1-c}=\dfrac{\lambda}{3\alpha-\lambda-1}<r_1.
  \end{equation*}
  Therefore $zf'(z)/f(z)$ lies in $\Delta$ and hence $f\in\mathcal{S}^*(\phi).$
\end{proof}
\begin{theorem}
  The class $G_{\lambda,\alpha}\;(\lambda>0,1/3<\alpha<1)$ satisfies the following inclusion relations:
  \begin{itemize}
    \item[(i)]$G_{\lambda,\alpha}\subset\mathcal{S}^*_{SG},$ whenever $2\lambda e<(e-1)(3\alpha-1)$
    \item[(ii)] $G_{\lambda,\alpha}\subset\mathcal{S}^*_{e},$ whenever $(2e-1)\lambda <(e-1)(3\alpha-1)$
    \item[(iii)] $G_{\lambda,\alpha}\subset\mathcal{S}^*_{S},$ whenever $(1+\sin(1))\lambda e<(1+\sin(1))(3\alpha-1)$
    \item[(iv)] $G_{\lambda,\alpha}\subset\mathcal{S}^*_{L},$ whenever $\sqrt{2}\lambda <(\sqrt{2}-1)(3\alpha-1)$
    \item[(v)] $G_{\lambda,\alpha}\subset\mathcal{S}^*_{Ne},$ whenever $5\lambda <2(3\alpha-1)$
    \item[(vi)]$G_{\lambda,\alpha}\subset\mathcal{S}^*_{C},$ whenever $5\lambda <2(3\alpha-1)$
    \item[(vii)]$G_{\lambda,\alpha}\subset\mathcal{S}^*_{Cr},$ whenever $(3-\sqrt{2})\lambda <(2-\sqrt{2})(3\alpha-1)$
    \item[(viii)]$G_{\lambda,\alpha}\subset\mathcal{S}^*_{P},$ whenever $(e+1)\lambda <(3\alpha-1).$
  \end{itemize}
\end{theorem}
\begin{proof}
  For different choices of $\phi$ with respective values of $r_1$(refer to Table~\ref{tab:template}), we apply Lemma~\ref{incgen} and the result follows directly.
  \begin{table}[H]
\parbox{.45\linewidth}{
\centering
\renewcommand{\arraystretch}{2.5}
\begin{tabular}{ c c c c}
\hline
$\phi$ & $\mathcal{S}^*(\phi)$ & $r_1$  & Reference \\ \hline

$\dfrac{2}{1+e^{-z}}$ & $\mathcal{S}^*_{SG}$ & $\dfrac{e-1}{e+1}$  & \cite{first} \\

 $e^z$& $\mathcal{S}^*_{e}$ & $1-\dfrac{1}{e}$ & \cite{mendiratta} \\
$1+\sin{z}$& $\mathcal{S}^*_{S}$& $\sin{1}$ & \cite{vktsine} \\
$\sqrt{1+z}$& $\mathcal{S}^*_{L}$ & $\sqrt{2}-1$ & \cite{sokol}\\ \hline
\end{tabular}
}
\hfill
\parbox{.45\linewidth}{
\centering
\renewcommand{\arraystretch}{2.5}
\begin{tabular}{ c c c c}
\hline
$\phi$ & $\mathcal{S}^*(\phi)$ & $r_1$  & Reference \\ \hline
$1+z-\dfrac{z^3}{3}$& $\mathcal{S}^*_{Ne}$ & $\dfrac{2}{3}$ & \cite{lateef} \\
$1+\dfrac{4}{3}z+\dfrac{2}{3}z^2$ & $\mathcal{S}^*_{C}$ & $\dfrac{2}{3}$ &\cite{cardiod} \\
$z+\sqrt{1+z^2}$ & $\mathcal{S}^*_{Cr}$& $2-\sqrt{2}$ &\cite{crescent} \\
$1+ze^z$ & $\mathcal{S}^*_{\wp}$& $\dfrac{1}{e}$& \cite{kamal}\\ \hline
\end{tabular}
}
\caption{Radii of the smallest disk with center 1, inscribed in $\mathcal{S}^*(\phi)$   }
\label{tab:template}
\end{table}
\end{proof}
\begin{theorem}\label{theorad}
If $f\in\Omega,$ then $f\in G_{\tiny{\frac{1}{2}},\tiny{\frac{1}{2}}}$ in the disc $|z|<r_0,$ where $r_0\approx 0.430496$ is the smallest positive root of $55r^{12}-28 r^{11}-854 r^{10}+148 r^9+2969 r^8-212 r^7-4286 r^6+28 r^5+2875 r^4+96 r^3-888 r^2-32 r+96=0.$	
\end{theorem}
\begin{proof}
Let $f\in\Omega,$ then $f$ can be written in the from~\eqref{t15}. Now if we let $\omega(z)=\int_0^z \varphi(\zeta)d\zeta,$ then clearly $\omega(z)$ and $\omega'(z)$ are analytic in $\mathbb{D}$ and we  can write $f$ as
\begin{equation}\label{t16}
	f(z)=z+\dfrac{1}{2}z\omega(z).
\end{equation}
Now by using the properties of $\varphi$ we have
\begin{equation*}
	|\omega(z)|=\left|\int_0^z \varphi(\zeta)d\zeta\right|\leq\int_0^z |\varphi(\zeta)|d\zeta\leq |z|
\end{equation*}
and
\begin{equation*}
	|\omega'(z)|=|\varphi(z)|\leq1.
\end{equation*}
Using Schwarz-Pick Lemma, we have for $z\in\mathbb{D},$
\begin{equation}\label{t17}
	|\omega''(z)|\leq\dfrac{1-|\omega'(z)|^2}{1-|z|^2}.
\end{equation}
In~\cite{dieudonne}, Dieudonne proved certain results which yield the following inequalities
\begin{equation}\label{t18}
	|\omega'(z)|\geq \dfrac{(|\omega(z)|-r^2)(1+|\omega(z)|)}{r(1-r^2)}
\end{equation}
and
\begin{equation}\label{t19}
	|z\omega'(z)-\omega(z)|\leq\dfrac{r^2-|\omega(z)|^2}{1-r^2}
\end{equation}
on $|z|=r,$ where $|\omega(z)|\leq r.$ In view of~\eqref{t16}, we obtain
\begin{eqnarray*}
\left|\dfrac{1+\tfrac{zf''(z)}{f'(z)}}{\tfrac{zf'(z)}{f(z)}}-1\right|&=& \left|\dfrac{z \left(z (\omega(z)+2) \omega''(z)-z \omega'(z)^2+(\omega(z)+2) \omega'(z)\right)}{\left(z \omega'(z)+\omega(z)+2\right)^2}\right|\\
&\leq& \dfrac{r((2+|z\omega'(z)-\omega(z)|)|\omega'(z)|)+r|\omega''(z)|(2+|\omega(z)|)}{(2(1-|\omega(z)|)-|z\omega'(z)-\omega(z)|)^2}.
\end{eqnarray*}
Using the inequalities~\eqref{t17},~\eqref{t18} and~\eqref{t19}, we get
\begin{eqnarray*}
\left|\dfrac{1+\tfrac{zf''(z)}{f'(z)}}{\tfrac{zf'(z)}{f(z)}}-1\right|&\leq& \dfrac{r}{\left(2(1-|\omega(z)|)-\left(\dfrac{r^2-|\omega(z)|^2}{1-r^2}\right)\right)^2}\Bigg[\left(2+\left(\dfrac{r^2-|\omega(z)|^2}{1-r^2}\right)\right)\left(\dfrac{1-|\omega(z)|^2}{1-r^2}\right)\\
& &+r\left(\dfrac{1-\left(\dfrac{(|\omega(z)|-r^2)(1+|\omega(z)|)}{r(1-r^2)}\right)^2}{1-r^2}\right)(2+|\omega(z)|)\Bigg].\\
\end{eqnarray*}
Writing $|\omega(z)|=\omega,$ we get
\begin{eqnarray*}
\left|\dfrac{1+\tfrac{zf''(z)}{f'(z)}}{\tfrac{zf'(z)}{f(z)}}-1\right|&\leq&\dfrac{1}{(1-r^2)\left(2 r^2 \omega-3 r^2+\omega^2-2 \omega+2\right)^2}(r^6 \omega+2 r^6-r^5 \omega^2+r^5-r^4 \omega^3\\
& &-4 r^4 \omega^2-7 r^4 \omega-6 r^4-r^3 \omega^4+4 r^3 \omega^2-3 r^3+2 r^2 \omega^4+8 r^2 \omega^3+10 r^2 \omega^2\\
& &+5 r^2 \omega+2 r^2+r \omega^4-3 r \omega^2+2 r-\omega^5-4 \omega^4-5 \omega^3-2 \omega^2).
\end{eqnarray*}
For $f$ to be in $G_{\tiny{\frac{1}{2}},\tiny{\frac{1}{2}}},$ it suffices to show that
\begin{eqnarray*}
& &\dfrac{1}{(1-r^2)\left(2 r^2 \omega-3 r^2+\omega^2-2 \omega+2\right)^2}(r^6 \omega+2 r^6-r^5 \omega^2+r^5-r^4 \omega^3-4 r^4 \omega^2-7 r^4 \omega-6 r^4\\
& &-r^3 \omega^4+4 r^3 \omega^2-3 r^3+2 r^2 \omega^4+8 r^2 \omega^3+10 r^2 \omega^2+5 r^2 \omega+2 r^2+r \omega^4-3 r \omega^2+2 r-\omega^5\\
& &-4 \omega^4-5 \omega^3-2 \omega^2)<1,
\end{eqnarray*}
which is equivalent to
\begin{eqnarray*}
\Phi(\omega,r)&:= &\omega^5+\left(r^3-3 r^2-r+5\right) \omega^4+\left(1-3 r^4\right) \omega^3+(-4 r^6+r^5+22 r^4-4 r^3\\
& &-32 r^2+3 r+10) \omega^2+\left(11 r^6-25 r^4+23 r^2-8\right) \omega-11 r^6-r^5+27 r^4+3 r^3\\
& &-18 r^2-2 r+4>0.
\end{eqnarray*}
We may note that $\omega=|\omega(z)|\leq|z|=r,$ so we have $0\leq\omega\leq r.$ Let us write
\begin{eqnarray*}
A&=& -4 r^6+r^5+22 r^4-4 r^3-32 r^2+3 r+10\\
B&=& 11 r^6-25 r^4+23 r^2-8\\
C&=& -11 r^6-r^5+27 r^4+3 r^3-18 r^2-2 r+4,
\end{eqnarray*}
then $B^2-4AC<0,$ whenever $r<r_1\approx0.430496.$ Also $A>0,$ whenever $r<r_2\approx0.565244.$ Thus
\begin{eqnarray*}
& &(-4 r^6+r^5+22 r^4-4 r^3-32 r^2+3 r+10) \omega^2+\left(11 r^6-25 r^4+23 r^2-8\right) \omega-11 r^6-r^5\\
& &+27 r^4+3 r^3-18 r^2-2 r+4>0,	
\end{eqnarray*}	
whenever $r<\min\{r_1,r_2\}=r_1.$ Next we observe that coefficients of $\omega^5$ and $\omega^4$ are always positive and coefficient of $\omega^3$ is positive for the range $0\leq r < r_3 =(1/3)^{1/4}\approx0.759836.$ It can be easily concluded that
\begin{equation*}
\Phi(\omega,r)>0\quad\text{whenever}\quad r<r_0=\min\{r_1,r_2,r_3\}=r_1.	
\end{equation*}
Hence the result.
\end{proof}
\begin{theorem}
If $f\in\mathcal{S}_e^*,$ then $f\in\Omega$ in the disc $|z|<r_0,$ where $r_0\approx 0.476813$ is the smallest positive root of $2(e^r-1)f_0(r)-1=0,$ where	
\begin{equation}\label{t25e}
  f_0(z)=z\exp{\left(\int_0^z\dfrac{e^t-1}{t}dt\right)}=z+z^2+\dfrac{3z^3}{4}+\dfrac{17z^4}{36}+\dfrac{19z^5}{72}+\cdots.
\end{equation}
Moreover, this estimate is sharp.
\end{theorem}
\begin{proof}
Let $f\in\mathcal{S}_e^*.$ Then $zf'(z)/f(z)\prec e^z,$ which further implies that
\begin{equation*}
  \left|\dfrac{zf'(z)}{f(z)}-1\right|\leq \max_{|z|=r}|e^{re^{i\theta}}-1|=e^r-1.
\end{equation*}
We apply~\cite[Theorem~2.7]{mendiratta} on $f$ and obtain that $|f(z)|\leq f_0(r)\;(|z|=r),$ where $f_0$ is given by~\eqref{t25e}. So
\begin{equation*}
  \left|zf'(z)-f(z)\right|\leq |f(z)|\left|\dfrac{zf'(z)}{f(z)}-1\right|\leq f_0(r)(e^r-1)\quad\text{on}\;|z|=r.
\end{equation*}
 Taking $|z|<r_0,$ we have $|zf'(z)-f(z)|<1/2$ and for the function $f_0,$ the inequality holds only in the disk $|z|<r_0,$ therefore result is sharp.
\end{proof}
\begin{theorem}
If $f\in\mathcal{S}_{Cr}^*,$ then $f\in\Omega$ in the disc $|z|<r_0,$ where $r_0\approx 0.485894$ is the smallest positive root of $2(r+\sqrt{1+r^2}-1)f_0(r)-1=0,$ where	
\begin{equation}\label{t25cr}
  f_0(z)=z\exp{\left(\int_0^z\dfrac{t+\sqrt{1+t^2}-1}{t}dt\right)}=z+z^2+\dfrac{3z^3}{4}+\dfrac{5z^4}{12}+\dfrac{z^5}{6}+\cdots.
\end{equation}
This result is sharp.
\end{theorem}
\begin{proof}
Let $f\in\mathcal{S}_{Cr}^*.$ Then $zf'(z)/f(z)\prec z+\sqrt{1+z^2},$ which is sufficient to say that
\begin{equation*}
  \left|\dfrac{zf'(z)}{f(z)}-1\right|\leq \max_{|z|=r}|e^{i\theta}+\sqrt{1+r^2e^{2i\theta}}-1|=r+\sqrt{1+r^2}-1.
\end{equation*}
By using~\cite[Theorem~1]{crescent}, we obtain $|f(z)|\leq f_0(r)\;(|z|=r),$ where $f_0$ is given by~\eqref{t25cr}. So on $|z|=r,$ we have
\begin{equation*}
  \left|zf'(z)-f(z)\right|\leq |f(z)|\left|\dfrac{zf'(z)}{f(z)}-1\right|\leq f_0(r)(r+\sqrt{1+r^2}-1),
\end{equation*}
 which is less that 1/2, provided $r<r_0.$ For the function $f_0,$ the inequality holds only in the disk $|z|<r_0,$ therefore result is sharp.
\end{proof}
\begin{theorem}
If $f\in\mathcal{S}_{SG}^*,$ then $f\in\Omega$ in the disc $|z|<r_0,$ where $r_0\approx 0.799269$ is the smallest positive root of $2\tan{(r/2)}f_0(r)-1=0,$ where	
\begin{equation}\label{t25sg}
  f_0(z)=z\exp{\left(\int_0^z\dfrac{e^t-1}{t(e^t+1)}dt\right)}=z+\dfrac{z^2}{2}+\dfrac{z^3}{8}+\dfrac{z^4}{144}-\dfrac{5z^5}{1152}+\cdots.
\end{equation}
\end{theorem}
\begin{proof}
Let $f\in\mathcal{S}_{SG}^*,$ so we have $zf'(z)/f(z)\prec 2/(1+e^{-z}).$ Therefore
\begin{equation*}
  \left|\dfrac{zf'(z)}{f(z)}-1\right|\leq \max_{|z|=r}\left|\dfrac{e^{re^{i\theta}}-1}{e^{re^{i\theta}}-1}\right|=\tan{(r/2)}.
\end{equation*}
Applying~\cite[Theorem~1.1]{first}, we have $|f(z)|\leq f_0(r)\;(|z|=r),$ where $f_0$ is given by~\eqref{t25sg}. So on $|z|=r$
\begin{equation*}
  \left|zf'(z)-f(z)\right|\leq |f(z)|\left|\dfrac{zf'(z)}{f(z)}-1\right|\leq f_0(r)\tan{(r/2)}<1/2,
\end{equation*}
whenever $r<r_0.$ Hence the result.
\end{proof}
\begin{theorem}
If $f\in\mathcal{S}_S^*,$ then $f\in\Omega$ in the disc $|z|<r_0,$ where $r_0\approx 0.531721$ is the smallest positive root of $2\sinh{1}f_0(r)-1=0,$ where	
\begin{equation}\label{t25}
  f_0(z)=z\exp{\left(\int_0^z\dfrac{\sin{t}}{t}dt\right)}=z+\dfrac{z^2}{2}+\dfrac{z^3}{8}+\dfrac{z^4}{144}-\dfrac{5z^5}{1152}+\cdots.
\end{equation}
\end{theorem}
\begin{proof}
Let $f\in\mathcal{S}_S^*.$ Then $zf'(z)/f(z)\prec 1+\sin{z},$ which further implies that
\begin{equation*}
  \left|\dfrac{zf'(z)}{f(z)}-1\right|\leq \max_{|z|=r}|\sin{re^{i\theta}}|=\sinh{r}.
\end{equation*}
Now by using the growth theorem given for $\mathcal{S}_S^*$ in~\cite{vktsine}, we have $|f(z)|\leq f_0(r)\;(|z|=r),$ where $f_0$ is given by~\eqref{t25}. Therefore
\begin{equation*}
  \left|zf'(z)-f(z)\right|\leq |f(z)|\left|\dfrac{zf'(z)}{f(z)}-1\right|\leq f_0(r)\sinh{r}\quad\text{on}\;|z|=r,
\end{equation*}
Hence $|zf'(z)-f(z)|\leq \sinh{1} f_0(r)<1/2,$ provided $r<r_0.$ Hence the result.
\end{proof}
\begin{theorem}
If $f\in\mathcal{S}_{\wp}^*,$ then $f\in \Omega$ in the disc $|z|<r_0,$ where $r_0\approx 0.43384$ is the smallest positive root of $2 r^2e^{e^r+r-1}-1=0.$ This result is sharp
\end{theorem}
\begin{proof}
Let $f\in\mathcal{S}_{\wp}^*,$ then we have $zf'(z)/f(z)\prec 1+ze^z,$ which further implies
\begin{equation*}
  \left|\dfrac{zf'(z)}{f(z)}-1\right|\leq \max_{|z|=r}|ze^z|=\max_{0\leq\theta<2\pi}|re^{i\theta}e^{re^{i\theta}}|=re^r.
\end{equation*}
Now by using~\cite[Theorem~2.2(ii)]{mendiratta}, we get $|f(z)|\leq re^{e^r-1}$ on $|z|=r.$ Finally we have
\begin{equation*}
  \left|zf'(z)-f(z)\right|=|f(z)|\left|\dfrac{zf'(z)}{f(z)}-1\right|\leq re^{e^r-1}(re^r)\quad\text{on}\;|z|=r.
\end{equation*}
Since $r<r_0,$ $|zf'(z)-f(z)|<1/2$ and thus $f\in\Omega.$ The result is sharp as for the function $f_0(z)=ze^{e^z-1},$ the inequality holds only in the disk $|z|<r_0.$
\end{proof}
\begin{theorem}
If $f\in\mathcal{S}_{RL}^*,$ then $f\in \Omega$ in the disc $|z|<r_0,$ where $r_0\approx 0.768$ is the smallest positive root of $2(\phi_0(-r)-1)f_0(r)-1=0,$ where
\begin{equation}\label{phi0}
 \phi_0(z)=\sqrt{2}-(\sqrt{2}-1)\sqrt{\dfrac{1-z}{1+2(\sqrt{2}-1)z}},
\end{equation}
and
\begin{equation}\label{f0rl}
  f_0(z)=z\left(\dfrac{\sqrt{1-z}+\sqrt{1+2(\sqrt{2}-1)}z}{2}\right)^{2(\sqrt{2}-1)}\exp{(p_0(z))}
\end{equation}
with
\begin{equation*}
  p_0(z)=\sqrt{2(\sqrt{2}-1)}\tan^{-1}\left(\sqrt{2(\sqrt{2}-1)}\left(\dfrac{\sqrt{1+2(\sqrt{2}-1)}z-\sqrt{1-z}}{\sqrt{1+2(\sqrt{2}-1)}z+2(\sqrt{2}-1)\sqrt{1-z}}\right)\right).
\end{equation*}
\end{theorem}
\begin{proof}
Let $f\in\mathcal{S}_{RL}^*,$ then $zf'(z)/f(z)\prec \phi_0(z)$ (given by~\eqref{phi0}).
So on $|z|=r$
\begin{equation*}
  \left|\dfrac{zf'(z)}{f(z)}-1\right|\leq\max_{0\leq\theta<2\pi}|\phi_0(re^{i\theta})-1|=\sqrt{2}-(\sqrt{2}-1)\sqrt{\dfrac{1+r}{1-2(\sqrt{2}-1)r}}-1
\end{equation*}
By using~\cite[Theorem~2.2(ii)]{mendiratta2}, we get $|f(z)|\leq |f_0(r)|$ on $|z|=r,$ where $f_0$ is given by~\eqref{f0rl}. Therefore
\begin{equation*}
  \left|zf'(z)-f(z)\right|=|f(z)|\left|\dfrac{zf'(z)}{f(z)}-1\right|\leq (\phi_0(-r)-1)f_0(r)<\dfrac{1}{2},
\end{equation*}
provided $r<r_0.$ Hence the result.
\end{proof}
\begin{theorem}
If $f\in\mathcal{S}_{L}^*,$ then $f\in \Omega$ in the disc $|z|<r_0,$ where $r_0\approx 0.734453$ is the positive root of	
\begin{equation}\label{t28}
  8r(1-\sqrt{1-r})\exp{(2\sqrt{1+r}-2)}-(1+\sqrt{1+r})^2=0.
\end{equation}
\end{theorem}
\begin{proof}
Let $f\in\mathcal{S}_{L}^*,$ which implies that $zf'(z)/f(z)\prec \sqrt{1+z}.$ Thus on $|z|=r,$ we have
\begin{equation*}
  \left|\dfrac{zf'(z)}{f(z)}-1\right|\leq \max_{|z|=r}|\sqrt{1+z}-1|=\max_{0\leq \theta<2\pi}|\sqrt{1+re^{i\theta}}-1|=1-\sqrt{1-r}.
\end{equation*}
Applying the growth theorem on $f$, we obtain
\begin{equation*}
 |f(z)|\leq \dfrac{4r\exp{(2\sqrt{1+r}-2)}}{(1+\sqrt{1+r})^2} ,\quad\text{on}\;|z|=r.
\end{equation*}
We observe that
\begin{equation*}
  |zf'(z)-f(z)|=|f(z)|\left|\dfrac{zf'(z)}{f(z)}-1\right|\leq \dfrac{4r(1-\sqrt{1-r})\exp{(2\sqrt{1+r}-2)}}{(1+\sqrt{1+r})^2} ,
\end{equation*}
which is less that $1/2,$ provided $r<r_0.$ Therefore $f\in\Omega.$
\end{proof}
\begin{theorem}
If $f\in\mathcal{S}_{Ne}^*,$ then $f\in \Omega$ in the disc $|z|<r_0,$ where $r_0\approx 0.524752$ is the positive root of 	
\begin{equation*}
 2r\left(r+\dfrac{r^3}{3}\right)\exp{\left(r-\dfrac{r^3}{9}\right)}=0
\end{equation*}
\end{theorem}
\begin{proof}
If $f\in\mathcal{S}_{Ne}^*,$ then $zf'(z)/f(z)\prec 1+z-z^3/3.$ We know that on $|z|=r$
\begin{equation*}
  \left|\dfrac{zf'(z)}{f(z)}-1\right|\leq \max_{|z|=r}\left|z-\dfrac{z^3}{3}\right|=\max_{0\leq\theta<2\pi}\left|re^{i\theta}-\dfrac{r^2e^{3i\theta}}{3}\right|=r+\dfrac{r^3}{3},
\end{equation*}
The growth theorem for the class $\mathcal{S}^*_{Ne}$ implies that for any $f\in\mathcal{S}^*_{Ne},$ $|f(z)|\leq|f_{Ne}(r)|$ on $|z|=r$, where
\begin{equation*}
  f_{Ne}(z)=z\exp{\left(z-\dfrac{z^3}{9}\right)}.
\end{equation*}
Using the above inequalities, we get
\begin{equation*}
  \left|zf'(z)-f(z)\right|=|f(z)|\left|\dfrac{zf'(z)}{f(z)}-1\right|\leq r\left(r+\dfrac{r^3}{3}\right)\exp{\left(r-\dfrac{r^3}{9}\right)}.
\end{equation*}
Since $r<r_0,$ we have $|zf'(z)-f(z)|<1/2$ and thus $f\in\Omega.$
\end{proof}
\begin{theorem}
If $f\in\mathcal{S}_{C}^*,$ then $f\in \Omega$ in the disc $|z|<r_0,$ where $r_0\approx 0.411914$ is the positive root of \begin{equation*}
2r e^{\tiny{\frac{r^2}{3}}+\tiny{\frac{4 r}{3}}} \left(\dfrac{2 r^2}{3}+\dfrac{4 r}{3}\right)-1=0.
\end{equation*}
\end{theorem}
\begin{proof}
Let $f\in\mathcal{S}_{C}^*.$ Then we have $zf'(z)/f(z)\prec 1+4z/3+2z^2/3,$ which gives
\begin{equation*}
  \left|\dfrac{zf'(z)}{f(z)}-1\right|\leq \left|\dfrac{4z}{3}+\dfrac{2z^2}{3}\right|.
\end{equation*}
Taking $z=re^{i\theta}\;(0\leq\theta<2\pi)$ it becomes
\begin{equation*}
  \left|\dfrac{zf'(z)}{f(z)}-1\right|\leq \left|\dfrac{4re^{i\theta}}{3}+\dfrac{2r^2e^{2i\theta}}{3}\right|=\dfrac{2}{3}\sqrt{4r^2+r^2+4r^3\cos{\theta}}\leq \dfrac{4r}{3}+\dfrac{2r^2}{3}.
\end{equation*}
Now by using the growth theorem for $\mathcal{S}^*_{C}$(refer to~\cite{cardiod}), we obtain $|f(z)|\leq|f_0(r)|$ on $|z|=r,$ where
\begin{equation*}
  f_0(z)=z\exp{\left(\dfrac{4z}{3}+\dfrac{z^2}{3}\right)}.
\end{equation*}
We observe that
\begin{equation*}
  \left|zf'(z)-f(z)\right|=|f(z)|\left|\dfrac{zf'(z)}{f(z)}-1\right|\leq r e^{\tiny{\frac{r^2}{3}}+\tiny{\frac{4 r}{3}}} \left(\dfrac{2 r^2}{3}+\dfrac{4 r}{3}\right)<\dfrac{1}{2},
\end{equation*}
provided $r<r_0.$ We may note that for $f(z)=f_0(z)$ the inequality $|zf'(z)-f(z)|<1/2$ holds only in the disk $|z|<r_0$ and thus the result is sharp.
\end{proof}
\begin{theorem}
Let $f\in\mathcal{S}^*(\phi_i)\;(i=1,2,3),$ then $f\in G_{\tiny{\frac{1}{2}},\tiny{\frac{1}{2}}}$ in the disk $|z|<r_i\;(i=1,2,3)$ for the following cases:
\begin{itemize}
\item [(i)] $\phi_1(z)=e^z$ and $r_1\approx0.537561$ is the smallest positive root of $e^r(1+r^2)^2-4(1-r^2)=0.$
\item [(ii)] $\phi_2(z)=\sqrt{1+z}$ and $r_2\approx0.429874$ is the smallest positive root of $(1+r^2)^2-4(1-r)^{3/2}(1-r^2)=0.$
\item [(iii)]$\phi_3(z)=2/(1+e^{-z})$ and $r_3\approx0.683447$ is the smallest positive root of $e^r(1+r^2)^2-8(1-r^2)=0.$
\end{itemize}
\end{theorem}
\begin{proof}
  Let $f\in\mathcal{S}^*(\phi),$ then we have $zf'(z)/f(z)\prec \phi(z).$ Thus there exists a Schwarz function $\omega$ with $\omega(0)=0$ and $|\omega(z)|\leq|z|$ such that
  \begin{equation*}
    \dfrac{zf'(z)}{f(z)}=\phi(\omega(z)),
  \end{equation*}
which further implies
\begin{equation*}
  \dfrac{1+zf''(z)/f'(z)}{zf'(z)/f(z)}-1=z\omega'(z)\dfrac{\phi'(\omega(z))}{\phi^2(\omega(z))}.
\end{equation*}
For $f$ to be in $G_1,$ it is sufficient to show that $|z\omega'(z)\phi'(\omega(z))/\phi^2(\omega(z))|<1.$\\
(i) Let $\phi(z)=e^z,$ then by using Lemma~\ref{lem2}, we have
\begin{equation*}
  \left|z\omega'(z)\dfrac{\phi'(\omega(z))}{\phi^2(\omega(z))}\right|=\left|\dfrac{z\omega'(z)}{e^{\omega(z)}}\right|\leq \dfrac{e^r(1+r^2)^2}{4(1-r^2)},
\end{equation*}
which is less than 1 provided $r<r_1.$\\
(ii) Let $\phi(z)=\sqrt{1+z}.$ By Lemma~\ref{lem2} we have for $r<r_2,$
\begin{equation*}
  \left|z\omega'(z)\dfrac{\phi'(\omega(z))}{\phi^2(\omega(z))}\right|=\left|\dfrac{z\omega'(z)}{(1+\omega(z))^{3/2}}\right|\leq \dfrac{(1+r^2)^2}{4(1-r)^{3/2}(1-r^2)}<1.
\end{equation*}
(iii) Let $\phi(z)=2/(1+e^{-z}),$ then by using Lemma~\ref{lem2}, we obtain
\begin{equation*}
  \left|z\omega'(z)\dfrac{\phi'(\omega(z))}{\phi^2(\omega(z))}\right|=\left|\dfrac{z\omega'(z)}{2e^{\omega(z)}}\right|\leq \dfrac{e^r(1+r^2)^2}{8(1-r^2)},
\end{equation*}
which is less than 1 whenever $r<r_3.$
\end{proof}

\end{document}